\documentclass[11pt,a4paper,reqno]{article}

\usepackage{graphicx} % Required for inserting images
\usepackage{amsfonts,amsopn, amsmath, amsthm, amssymb}
\usepackage{color}
\usepackage{physics}
\usepackage{tikz}
\usepackage{soul}
\usepackage{float}

\newtheorem{theorem}{Theorem}[section]
\newtheorem{corollary}[theorem]{Corollary}
\newtheorem{lemma}[theorem]{Lemma}
\theoremstyle{definition}
\newtheorem{definition}[theorem]{Definition}
\newtheorem{proposition}[theorem]{Proposition}
\newtheorem{conjecture}[theorem]{Conjecture}

\newtheorem*{theorem*}{Theorem}
\newtheorem*{lemma*}{Lemma}

\usepackage[hidelinks]{hyperref}

\title{A graph energy conjecture through the lenses \\of semidefinite programming}

\author{Aida Abiad\thanks{\texttt{a.abiad.monge@tue.nl},\\ 
Department of Mathematics and Computer Science, Eindhoven University of Technology, The Netherlands \\ Department of Mathematics and Data Science, Vrije Universiteit Brussel, Belgium}
\quad
Gabriel Coutinho\thanks{\texttt{gabriel@dcc.ufmg.br},\\ Department of Computer Science, Federal University of Minas Gerais, Brazil} 
\quad
Emanuel Juliano\thanks{\texttt{emanuelsilva@dcc.ufmg.br},\\ Department of Computer Science, Federal University of Minas Gerais, Brazil}
\quad Luuk Reijnders\thanks{\texttt{l.e.r.m.reijnders@tue.nl},\\ Department of Mathematics and Computer Science, Eindhoven University of Technology, The Netherlands}}
\date{}

\begin{document}

\maketitle

\begin{abstract}
    Let $G$ be a graph on $n$ vertices with independence number $\alpha(G)$. Let $\mathcal{E}(G)$ be the energy of a graph, defined as the sum of the absolute values of the adjacency eigenvalues of $G$. Using Graffiti, Fajtlowicz conjectured in the 1980s that $$\frac{1}{2}\mathcal{E}(G) \geq n - \alpha(G).$$

    In this paper we derive a semidefinite program formulation of the graph energy, and we use it to obtain several results that constitute a first step towards proving this conjecture. In particular,  we show that 
    $$\frac{1}{2}\mathcal{E}(G) \geq n-\chi_f(\overline{G}) \quad \text{ and }\quad \frac{1}{2}\mathcal{E}(G) \geq n-H(G),$$ where $\chi_f(G)$ is the fractional chromatic number and $H(G)$ is Hoffman's ratio number. As a byproduct of the SDP formulation we obtain several lower bounds for the graph energy that improve and refine previous results by Hoffman (1970) and Nikiforov (2007). The later author showed that the conjecture holds for almost all graphs. However, the graph families known to attain the conjecture with equality are highly structured and do not represent typical graphs. Motivated by this, we prove the following bound in support of the conjecture for the class of highly regular graphs $$\frac{1}{2}\mathcal{E}(G) \geq n-\vartheta^-(G),$$
    where $\vartheta^-$ is Schrijver's theta number.\\

%\noindent \textbf{Keywords:} graph energy, independence number, conjecture
    
\end{abstract}

%%%%%%%%%%%%%%%%%%%%%%%%%%%%%%%%%%%%
\section{Introduction}
%%%%%%%%%%%%%%%%%%%%%%%%%%%%%%%%%%%%

Let $G$ be a graph on $n$ vertices and A be its adjacency matrix with eigenvalues $\lambda_1 \geq \dots \geq \lambda_n$. The \emph{energy} of $G$ is defined as %the sum over the absolute values of its eigenvalues
\begin{equation*}
    \mathcal{E}(G) = \sum_{i=1}^n |\lambda_i|.
\end{equation*}

The graph energy, which was originally introduced by Gutman in 1978 motivated by applications in molecular chemistry, has attracted a considerable amount of attention over the last decades, see e.g. \cite{akbari2019proof, gutman2001energy,VANDAM2014123}.

In this paper, we investigate the following spectral conjecture concerning graph energy, which was proposed by Fajtlowicz using Graffiti in the 1980s (c.f. \cite{Aouchiche2010}), and which recently appeared in a survey by Liu and Ning \cite{liu_unsolved_2023}:     
    \begin{conjecture}[\cite{liu_unsolved_2023}, Conjecture 20]\label{con:energy_lb} 
        Let $G$ be a graph on $n$ vertices with independence number $\alpha(G)$. Let $A$ be the adjacency matrix of $G$ with eigenvalues $\lambda_1\geq\cdots \geq \lambda_n$. Then
        \begin{equation*}
            \sum_{\lambda_i > 0} \lambda_i = \frac{1}{2}\mathcal{E}(G) \geq n-\alpha(G).
        \end{equation*}
    \end{conjecture}

Conjecture \ref{con:energy_lb} has been computationally verified for all graphs with up to 10 vertices \cite{liu_unsolved_2023} and shown to hold for almost all graphs \cite{nikiforov2007energy}.

Semidefinite programming (SDP) methods have been successful in tackling several graph theory problems, such as graph partitioning~\cite{sotirov2014efficient} and bounding graph parameters~\cite{laurent2024semidefinite}. More recently, these techniques have also been applied to variations of graph energy \cite{coutinho_conic_2024, zhang_extremal_2024}. 
Following this promising line of research, in this paper we propose to tackle Conjecture \ref{con:energy_lb} by first formulating the graph energy as an SDP (see SDP \eqref{eq:SDP_max}), and then using feasible solutions to this SDP to work towards proving the conjecture. 

Our first result is thus the following SDP formulation of the graph energy part of Conjecture \ref{con:energy_lb}.

\begin{theorem} \label{thm:SDP_max}
    Let $G$ be an $n$-vertex graph, then
    \begin{equation*}
        \frac{1}{2}  \mathcal{E}(G) = \sum_{\lambda_i > 0} \lambda_i = \max \{\langle A, X \rangle : I \succcurlyeq X \succcurlyeq 0\}.
    \end{equation*}
\end{theorem}

Our approach is to construct feasible solutions for the SDP formulation of Theorem \ref{thm:SDP_max} and use them to further support Conjecture \ref{con:energy_lb}. As a byproduct, each of our results yields a new lower bound on the graph energy, strengthening previously known bounds (see Andrade et al. \cite[Thm.~1.1]{andrade2017lower}, Day–So \cite[Thm.~3.4]{day2008graph}, Hoffman \cite[Eq.~6.1]{hoffman2003eigenvalues}, and Nikiforov \cite[Eq.~4]{nikiforov2007energy}).

One of the motivations for Conjecture \ref{con:energy_lb} is an old result of Hoffman \cite[Eq. 6.1]{hoffman2003eigenvalues}, which states that $$\frac{1}{2} \mathcal{E}(G) \geq n -  \chi(\overline{G}),$$ from which Conjecture \ref{con:energy_lb} follows for perfect graphs. Here $\chi(\overline{G})$ is the smallest number of cliques required to cover the vertices of $G$. Our second main result is a strengthening of the aforementioned Hoffman's result. We replace the $\chi(\overline{G})$ in Hoffman's result by $\chi_f(\overline{G})$ (recall that $\chi_f(\overline{G})$ is a relaxation of the chromatic number called the fractional chromatic number and satisfies $\chi_f(\overline{G}) \leq \chi(\overline{G})$).

\begin{theorem} \label{thm:energy_n_minus_chi_f}
    Let $G$ be a $n$-vertex graph. Then,
    \begin{equation*}
        \frac{1}{2} \mathcal{E}(G) \geq n -  \chi_f(\overline{G}).
    \end{equation*}
\end{theorem}

In particular, Theorem \ref{thm:energy_n_minus_chi_f} proves Conjecture \ref{con:energy_lb} for graphs which satisfy $\alpha(G) = \chi_f(\overline{G})$. Moreover, in order to show Theorem \ref{thm:energy_n_minus_chi_f} we derive a new lower bound for the graph energy based on a graph decomposition (Lemma \ref{lem:convex_comb_orth_matrices}). This result, which is of independent interest, extends previous work by Andrade, Robbiano and San Martin \cite{andrade2017lower}, and Day and So \cite{day2008graph}.

A celebrated result in spectral graph theory is Hoffman's \emph{ratio bound} on the independence number of a graph, which says that for a regular graph it holds that

\begin{equation}\label{eq:ratiobound}
\alpha(G)\leq H(G) := n \frac{-\lambda_n}{\lambda_1 - \lambda_n}.
\end{equation}

Our third main result relates the graph energy with the above ratio bound \eqref{eq:ratiobound} as follows.

\begin{theorem} \label{thm:energy_n_minus_ratio}
    Let $G \not \cong \overline{K_n}$ be a $n$-vertex regular graph. Then,
    \begin{equation*}
        \frac{1}{2} \mathcal{E}(G) \geq n -  H(G).
    \end{equation*}
\end{theorem}
In particular, Theorem \ref{thm:energy_n_minus_ratio} tells us Conjecture \ref{con:energy_lb} is true for all graphs which meet Hoffman's ratio bound. 

Based on the techniques used in Theorem \ref{thm:energy_n_minus_chi_f}, we prove the following lower bound on the graph energy.

\begin{theorem} \label{thm:energy_regular_scaled_complement}
    Let $G \not \cong K_n, \overline{K_n}$ be a $n$-vertex regular graph. Then,
    \begin{equation*}
        \frac{1}{2} \mathcal{E}(G) \geq \frac{2m - \lambda_1(\lambda_1 - \lambda_2)}{\lambda_2 - \lambda_n}.
    \end{equation*}
\end{theorem}

Theorem \ref{thm:energy_regular_scaled_complement} can be seen as a refinement of a lower bound due to Nikiforov \cite[Eq.~(4)]{nikiforov2007energy} for the class of regular graphs. As an application of this result, we show that Conjecture \ref{con:energy_lb} holds for the class of Johnson graphs (Corollary \ref{cor:Johnson}).

The \emph{Schrijver theta number} of a graph $G$, denoted by $\vartheta^-(G)$, is defined as
\begin{equation*}
    \vartheta^-(G) = \max \{\langle J, X \rangle : X  \succcurlyeq 0, X \geq 0, \langle I, X \rangle = 1, X\circ A = 0\}.
\end{equation*}
Schrijver's theta number relates to the other parameters in this paper as follows (see e.g. \cite{lovasz1979shannon}):
\begin{equation*}
    \alpha(G) \leq \vartheta^-(G) \leq \chi_f(\overline{G}) \leq \chi(\overline{G}),
\end{equation*}
and it also holds that $\vartheta^-(G) \leq H(G)$.

An argument supporting Conjecture \ref{con:energy_lb} is that it holds for almost all graphs \cite{nikiforov2007energy}. However, the asymptotic behavior of the energy of almost all graphs is quite different from the known graph classes that meet equality in Conjecture \ref{con:energy_lb} (e.g. complete balanced multipartite graphs and the Cartesian product $K_n \square K_2$). One possible explanation for this discrepancy is that the known cases of equality exhibit a high degree of symmetry, in contrast to random graphs~\cite{erdos1963asymmetric}. Motivated by this observation, in the following result we improve Theorem \ref{thm:energy_n_minus_chi_f} and Theorem \ref{thm:energy_regular_scaled_complement} within the class of highly regular graphs, introduced by de Carli Silva et al.~\cite{de2019algebras}. This graph class includes the well-known vertex-transitive graphs, distance-regular graphs, and 1-walk-regular graphs.

\begin{theorem} \label{thm:highly_regular}
Let $G$ be an $n$-vertex highly regular graph, then
\begin{equation*}
    \frac{1}{2}  \mathcal{E}(G) \geq n-\vartheta^-(G).    
\end{equation*}
\end{theorem}
As mentioned before, $\vartheta^-(G)$ is the closest to $\alpha(G)$ of all graph parameters we consider, in particular the inequality from Theorem \ref{thm:highly_regular} is stronger than those of Theorems \ref{thm:energy_n_minus_chi_f} and \ref{thm:energy_n_minus_ratio}.

This paper is structured as follows. In Section \ref{sec:SDP_formulation} we provide the SDP formulation of the graph energy from Theorem \ref{thm:SDP_max}. Next, in Section \ref{sec:frac_chrom} we show a bound on the graph energy in terms of a decomposition (Lemma \ref{lem:convex_comb_orth_matrices}), and use this result to prove Theorem \ref{thm:energy_n_minus_chi_f}. Then, in Section \ref{sec:adjacency_structure}, we look at a restricted version of SDP \eqref{eq:SDP_max}, and use its optimal solution to prove Theorems \ref{thm:energy_n_minus_ratio} and \ref{thm:energy_regular_scaled_complement}. In Section \ref{sec:high_reg}, we look at highly regular graphs and prove Theorem \ref{thm:highly_regular}. Finally, in Section \ref{sec:concludingremarks} we conclude and discuss a few open questions.

%%%%%%%%%%%%%%%%%%%%%%%%%%%%%%%%%%%%
\section{SDP formulation for the graph energy}\label{sec:SDP_formulation}
%%%%%%%%%%%%%%%%%%%%%%%%%%%%%%%%%%%%

    Throughout this section, let $G$ be a $n$ vertex graph with adjacency matrix $A$ and eigenvalues $\lambda_1 \geq \cdots \geq \lambda_n$.
    
    Our approach to Conjecture \ref{con:energy_lb} will be through a semidefinite program formulation of the energy of a graph. A \emph{semidefinite program} (SDP) is a generalization of a linear program, where the inequality constraints are replaced by positive semidefinite constraints. In particular, we are interested in the following two SDPs.

    \begin{equation}\label{eq:SDP_max}
        \begin{aligned}
        \text{maximize}  \ \ &\langle A,X \rangle \\
        \text{subject to}\ \ &X  \succcurlyeq 0, \\
        &I-X \succcurlyeq 0.
        \end{aligned}
    \end{equation}
    
    \begin{equation}\label{eq:SDP_min}
        \begin{aligned}
        \text{minimize}  \ \ &\langle I, Y \rangle \\
        \text{subject to}\ \ & Y \succcurlyeq 0, \\
        &Y-A \succcurlyeq 0.
        \end{aligned}
    \end{equation}

    \begin{proof}[Proof of Theorem \ref{thm:SDP_max}]
        We will prove this theorem by showing that any feasible solution for \eqref{eq:SDP_min} provides an upper bound on any feasible solution of \eqref{eq:SDP_max}. We then present one feasible solution for each, both having objective value $\frac12\mathcal{E}(G)$. 
        Let $X,Y$ be solutions for \eqref{eq:SDP_max} and \eqref{eq:SDP_min} respectively. Since $I-X,Y-A \succcurlyeq 0$, we have
        \begin{equation*}
            \langle X, Y-A \rangle \geq 0 \hspace{5mm} \text{ and } \hspace{5mm}\langle I-X,Y \rangle \geq 0.
        \end{equation*}
        This implies
        \begin{equation*}
            \langle A, X\rangle \leq \langle Y,X \rangle \leq \langle Y, I \rangle.
        \end{equation*}

        We now note that $X=\sum_{\lambda_i > 0} E_i$ is a feasible solution of \eqref{eq:SDP_max} and $Y= \sum_{\lambda_i > 0} \lambda_i E_i$ is a feasible solution of \eqref{eq:SDP_min} and these solutions satisfy $\langle A, X\rangle = \sum_{\lambda_i>0} \lambda_i = \langle I,Y\rangle$.
    \end{proof}
    
    The important takeaway from this, is that any feasible $X$ of SDP \eqref{eq:SDP_max} satisfies
    \begin{equation*}
        \frac{1}{2}\mathcal{E}(G) = \sum_{\lambda_i > 0} \lambda_i \geq \langle A,X \rangle.
    \end{equation*}
    Thus, if we can find a matrix $X$ with eigenvalues between $0$ and $1$ for which $\langle A,X \rangle = n-\alpha(G)$, then this would prove Conjecture \ref{con:energy_lb}.

    In the next sections we present different approaches for finding feasible matrices for SDP \eqref{eq:SDP_max}. Although we are not capable of proving the full conjecture, each result can be seen as a new lower bound for the graph energy that improves on existing bounds in the literature by Hoffman \cite{hoffman2003eigenvalues} and Nikiforov \cite{nikiforov2007energy}, and thus they are of independent interest.

%%%%%%%%%%%%%%%%%%%%%%%%%%%%%%%%%%%%%%%%%%%%
\section{Graph decomposition and fractional chromatic number}\label{sec:frac_chrom}
%%%%%%%%%%%%%%%%%%%%%%%%%%%%%%%%%%%%%%%%%%%%

In 1970, Hoffman \cite{hoffman2003eigenvalues} proved the following result (also present in \cite{thompson1975singular}).
\begin{lemma}[\cite{hoffman2003eigenvalues}]\label{lem:Hoffman_energy}
    Let $G=(V,E)$ be a graph and let $S_1, \dots, S_t \subseteq V$ be a partition of $V$. Then,
    \begin{equation*}
        \mathcal{E}(G) \geq \sum_i \mathcal{E}(G[S_i]).
    \end{equation*}
\end{lemma}
Recall that the parameter $\chi(\overline{G})$ is the smallest number of cliques required to cover $V$. Using the fact that the energy of a complete graph is $\mathcal{E}(K_k) = 2(k-1)$ and taking $\chi(\overline{G})$ disjoint cliques, Lemma \ref{lem:Hoffman_energy} implies that
\begin{equation} \label{eq:energy_chi}
    \frac{1}{2}\mathcal{E}(G) \geq \sum_i (|S_i|-1) = n - \chi(\overline{G}).
\end{equation}

In order to improve inequality \eqref{eq:energy_chi}, we generalize Lemma \ref{lem:Hoffman_energy} as follows. 

\begin{lemma} \label{lem:convex_comb_orth_matrices}
    Let $M$ be a symmetric $n\times n$ matrix. Let $Q_1, \dots, Q_t$ be a sequence of real matrices, each $Q_i$ of dimension $n \times k_i$. Let $y_1, \dots, y_t$ be a sequence of non-negative real numbers such that $\sum_{i=1}^t y_i Q_i Q_i^T = I$. Then
    \begin{equation*}
        \mathcal{E}(M) \geq \sum_{i=1}^t y_i \mathcal{E}\left(Q_i^T M Q_i \right).
    \end{equation*}
\end{lemma}
\begin{proof}
    Let $\lambda_1 \geq \cdots \geq  \lambda_n$ be the eigenvalues of $M$. Denote by $\mathcal{E}^+(M) = \sum_{\lambda_i > 0} \lambda_i$ the positive energy of $M$. Analogously to Theorem \ref{thm:SDP_max}, we can write $\mathcal{E}^+(M)$ as an SDP:
    \begin{equation}\label{eq:SDP_max_general_M}
        \mathcal{E}^+(M) = \max \{\langle M, X \rangle : I \succcurlyeq X \succcurlyeq 0\}.
    \end{equation}

    Let $X_i$ be the optimal solution for SDP \eqref{eq:SDP_max_general_M} for the matrix $Q_i^T M Q_i$. That is, $X_i$ is a real $k_i \times k_i$ matrix such that $I_{k_i} \succcurlyeq X_i \succcurlyeq 0$ and $\langle Q_i^T M Q_i, X_i \rangle = \mathcal{E}^+(Q_i^T M Q_i)$. 

    Define $X := \sum_{i} y_i Q_i X_i Q_i^T$. We claim that $X$ is feasible for SDP \eqref{eq:SDP_max_general_M} with objective value $\sum_{i} y_i \mathcal{E}^+\left(Q_i^T M Q_i \right)$. Indeed, since $X$ is the sum of positive semidefinite matrices, $X$ itself is positive semidefinite. Moreover, 
     \begin{equation*}
        I - X = \sum_{i} y_i Q_iQ_i^T - \sum_{i} y_i Q_iX_iQ_i^T = \sum_{i} y_i Q_i(I_{k_i} - X_i)Q_i^T \succcurlyeq 0,
     \end{equation*}
    so $X$ is feasible. Now we compute the objective value:
    \begin{equation*}
        \langle M, X \rangle = \sum_{i} y_i \langle M, Q_i X_i Q_i^T \rangle = \sum_{i} y_i \langle Q_i^T M Q_i, X_i\rangle = \sum_{i} y_i \mathcal{E}^+\left(Q_i^T M Q_i\right).
    \end{equation*}
    Since $\langle M, X\rangle$ is a lower bound for $\mathcal{E}^+(M)$, we conclude that 
    \begin{equation*}
         \mathcal{E}^+(M) \geq \sum_{i} y_i \mathcal{E}^+\left(Q_i^T M Q_i \right).
    \end{equation*}
    Finally, in order to achieve our desired bound for the energy, we note that $\mathcal{E}^+(M) = \frac{1}{2}(\mathcal{E}(M)+ \langle M, I\rangle)$ and $\langle M, I\rangle = \sum_i y_i \left\langle Q_i^TMQ_i, I\right\rangle$.  
\end{proof}

The applicability of the Lemma \ref{lem:convex_comb_orth_matrices} consists of correctly choosing matrices $Q_i$. In the case where each $Q_i$ is restricted to be a $0/1$-matrix, we get the following combinatorial generalization of Lemma \ref{lem:Hoffman_energy}.

\begin{lemma}
    \label{lem:fractional_subgraph_covering}
    Let $G = (V, E)$ be a graph, let $S_1, \dots, S_t \subseteq V$ and let $y_i \geq 0$ be real numbers satisfying $\sum_{S_i \ni v} y_i = 1 $ for all $v \in V$. Then,
    \begin{equation*}
        \mathcal{E}(G) \geq \sum_{i} y_i \mathcal{E}\left(G[S_i] \right).
    \end{equation*}
\end{lemma}
\begin{proof}
    Let $Q_i$ be the $0/1$-matrix with rows indexed by $V$ and columns indexed by $S_i$, with entry equal to $1$ exactly when the vertices coincide.
    
    Let $e_{S_i}$ be the $0/1$-indicator vector for the set $S_i$. Note that $Q_i^T Q_i = I_{k_i}$ and that $Q_iQ_i^T = \text{Diag}(e_{S_i})$. Therefore, we have that $\sum_{i} y_i Q_iQ_i^T = I$ and the result follows from applying Lemma \ref{lem:convex_comb_orth_matrices}.
\end{proof}

Let $\mathcal{C} \subset \mathcal{P}(V)$ be the set of cliques of $G$. We define the $n \times \abs{\mathcal{C}}$ matrix $M$ by $M_{iU} = 1$ if $v_i \in U$, and $0$ otherwise. The fractional chromatic number of $\overline{G}$, denoted by $\chi_f(\overline{G})$ is defined as
\begin{equation}\label{eq:chif_opt}
\chi_f(\overline{G}) = \min \{\langle\mathbf{1}, y\rangle : My = \mathbf{1}, y \geq 0 \}.
\end{equation}

Using Lemma \ref{lem:fractional_subgraph_covering}, we prove Theorem \ref{thm:energy_n_minus_chi_f}.
\begin{proof}[Proof of Theorem \ref{thm:energy_n_minus_chi_f}]
    Let $\mathcal{C} = \{C_1, \dots, C_t\}$ be the set of cliques of $G$ and let $y$ be a optimal solution for SDP \eqref{eq:chif_opt}, that is, $y$ satisfies $\sum_i y_i = \chi_f(\overline{G})$ and $\sum_{C_i \ni v} y_i= 1$ for every vertex $v$. By Lemma \ref{lem:fractional_subgraph_covering} and using the fact that $\mathcal{E}(K_k) = 2(k-1)$, we have that
    \begin{equation*}
    \frac{1}{2}\mathcal{E}(G) \geq \sum_{i} \frac{y_i}{2} \mathcal{E}(G[C_i]) = \sum_iy_i(|C_i| - 1) = n-\chi_f(\overline{G}). \qedhere
    \end{equation*}
\end{proof}

We end this section by mentioning another application of Lemma \ref{lem:fractional_subgraph_covering}.

\begin{corollary}\label{cor:hyper_energetic_vertex_transitive}
    Let $G = (V, E)$ be a vertex transitive graph. If there exists a set $S \subseteq V$ such that $\mathcal{E}(G[S]) \geq c|S|$, for some $c > 0$, then
    \begin{equation*}
        \mathcal{E}(G) \geq cn.
    \end{equation*}
\end{corollary}
\begin{proof}
    Let $H \cong G[S]$ and let $\mathcal{H} \subseteq \mathcal{P}(V)$ denote the set of vertices that induce a copy of $H$ in $G$. Since $G$ is vertex transitive, every vertex belongs to the same number of copies of $H$, say $c_H$. We apply Lemma \ref{lem:fractional_subgraph_covering} and obtain
    \begin{equation*}
        \mathcal{E}(G) \geq \frac{1}{c_H} \sum_{S_i \in \mathcal{H}} \mathcal{E}(G[S_i]) = \frac{|\mathcal{H}|}{c_H} \mathcal{E}(H) = \frac{n}{|S|} \mathcal{E}(H) \geq cn,
    \end{equation*}
    where the last equality follows from double counting the pairs $(v, S_i)$, with $S_i \in \mathcal{H}$ and $v \in S_i$.
\end{proof}

%%%%%%%%%%%%%%%%%%%%%%%%%%%%%%%%%%%%
\section{Adjacency structure and Hoffman's ratio bound}\label{sec:adjacency_structure}
%%%%%%%%%%%%%%%%%%%%%%%%%%%%%%%%%%%%

In this section, we consider a restriction of SDP \eqref{eq:SDP_max}, the optimal solution of which will allow us to prove Theorems \ref{thm:energy_n_minus_ratio} and \ref{thm:energy_regular_scaled_complement}.

Consider the following restriction of SDP \eqref{eq:SDP_max}.

    \begin{equation}\label{eq:SDP_adjacency}
        \begin{aligned}
        \text{maximize}  \ \ &\langle A,X \rangle \\
        \text{subject to}\ \ &X  \succcurlyeq 0, \\
        &I - X \succcurlyeq 0,\\
        & X = aA + b\overline{A} + cI, \\
        & a, b \in \mathbb{R}. 
        \end{aligned}
    \end{equation}
    Note that since $A$ has $0$ diagonal we can shift the eigenvalues of $X$ without changing the value of the product $\langle A,X \rangle$, by adding a multiple of the identity matrix $I$ to $X$. Hence, we can ignore $c$, as this will not affect the gap between the largest and the smallest eigenvalue. So in order to provide a feasible solution for SDP \eqref{eq:SDP_adjacency}, it suffices to find values of $a$ and $b$ such that the eigenvalues of $aA + b\overline{A}$ lie in an interval of size $1$.

\begin{lemma}\label{lem:energy_bound_adjacency}
    Let $G \not \cong \overline{K_n}$ be a graph on $n$ vertices and $m$ edges with adjacency eigenvalues $\lambda_1 \geq \dots \geq \lambda_n$, then
    \begin{equation*}
        \frac{1}{2}\mathcal{E}(G) \geq \frac{2m}{\lambda_1 - \lambda_n}.
    \end{equation*}
\end{lemma}
\begin{proof}
    Let $a = \frac{1}{\lambda_1 - \lambda_n}$. Note that the eigenvalues of the matrix $aA$ lie in an interval of size $1$ and, therefore, there exists $c$ such that $X = aA + cI$ is feasible with objective value $\langle A, X \rangle = a\langle A, A \rangle = \frac{2m}{\lambda_1 - \lambda_n}$.
\end{proof}

The choice of $a$ in Lemma \ref{lem:energy_bound_adjacency} is in fact optimal when assuming $b=0$, as $a = \frac{1}{\lambda_1 - \lambda_n}$ is the largest value such that the eigenvalues of $aA$ lie in an interval of size $1$.

Let $G \not \cong \overline{K_n}$ be a $k$-regular graph on $n$ vertices and $m$ edges with independence number $\alpha(G)$ and adjacency eigenvalues $k = \lambda_1 \geq \dots \geq \lambda_n$. Define the quantity $H(G) := n \frac{-\lambda_n}{\lambda_1 - \lambda_n}$, also known as Hoffman's ratio bound for the independence number:
\begin{equation*}
    \alpha(G) \leq H(G) = n \frac{-\lambda_n}{\lambda_1 - \lambda_n}.
\end{equation*}

Using Lemma \ref{lem:energy_bound_adjacency}, we can prove Theorem \ref{thm:energy_n_minus_ratio}.

\begin{proof}[Proof of Theorem \ref{thm:energy_n_minus_ratio}]
    It follows by a direct application of Lemma \ref{lem:energy_bound_adjacency}:
    \begin{equation*}
        \frac{1}{2} \mathcal{E}(G) \geq \frac{2m}{\lambda_1 - \lambda_n} = \frac{nk}{\lambda_1 - \lambda_n} = \frac{n(\lambda_1-\lambda_n) + n\lambda_n}{\lambda_1 - \lambda_n} = n - H(G). \qedhere
    \end{equation*}
\end{proof}

We remark that Hoffman's ratio bound $H(G)$ is incomparable with $\chi_f(\overline{G})$. Moreover, it is known that there are exponentially many graphs which meet the ratio bound with equality, see e.g. \cite{abiad2025hoffman,abiad2025hoffmanv2}.

We now consider the case where $b$ might be non-zero. 

\begin{proof}[Proof of Theorem \ref{thm:energy_regular_scaled_complement}]
    Let $\mathbf{v_1}, \hdots, \mathbf{v_n}$ be an orthonormal basis of eigenvectors of $A$ with eigenvector $\mathbf{v_i}$ associated with eigenvalue $\lambda_i$. In particular, since $G$ is regular, $\mathbf{v_1} = \frac{1}{\sqrt{n}} \mathbf{1}$. Hence, $\mathbf{v_i}$ are also eigenvectors of $aA + b\overline{A}$ corresponding to the following $n$ eigenvalues:
    \begin{equation*}
        (aA + b\overline{A})\mathbf{v_i}
        = \begin{cases}
            \left((a-b)\lambda_i + (n-1)b\right)\mathbf{v_i} &\text{ if } i=1, \\
            \left((a-b)\lambda_i - b\right)\mathbf{v_i} &\text{ else.}
        \end{cases}      
    \end{equation*}
    Denote these eigenvalues by $\mu'_1 , \hdots, \mu'_n$, with $\mu'_i$ corresponding to $\mathbf{v_i}$. Note that they are not necessarily ordered by size, unlike the $\lambda_i$. Denote by $\mu_1 \geq \cdots \geq \mu_n$ the \emph{ordered} eigenvalues of $aA + b\overline{A}$. Note that
    \begin{align*}
        \mu_1 &= \max\{\mu'_1, \mu'_2, \mu'_n\}, & 
        \mu_n &= \min\{\mu'_1, \mu'_2, \mu'_n\}.
    \end{align*}

    The feasible solutions for SDP \eqref{eq:SDP_adjacency} are exactly the values $a$ and $b$ such that the eigenvalues of $aA + b\overline{A}$ lie in an interval of size at most one. By choosing
    \begin{align*}
        a &= \frac{n + \lambda_2 - \lambda_1}{n(\lambda_2 - \lambda_n)}, & 
        b &= \frac{\lambda_2 - \lambda_1}{n(\lambda_2 - \lambda_n)},
    \end{align*}

    we have $\mu_1 = \mu_1'= \mu_2'$, and subsequently $\mu_n = \mu'_n$. To show that this solution is feasible, we only need to bound $\mu_1-\mu_n$:
    \begin{equation*}
        \mu_1-\mu_n = \mu_2'-\mu_n' =  (a-b)(\lambda_2-\lambda_n) = \frac{n(\lambda_2 - \lambda_n)}{n(\lambda_2 - \lambda_n)} \leq 1.
    \end{equation*}
    And finally, we compute the objective value:
    \begin{align*}
        \langle A, aA + b\overline{A} + cI\rangle &= a\langle A, A\rangle \\
        &= \frac{2m(n + \lambda_2 - \lambda_1)}{n(\lambda_2 - \lambda_n)} \\
        &= \frac{2m - \lambda_1(\lambda_1 - \lambda_2)}{\lambda_2 - \lambda_n}. \qedhere
    \end{align*}
\end{proof}

The choice of $a$ and $b$ used in the above proof of Theorem \ref{thm:energy_regular_scaled_complement} might seem arbitrary, but they are in fact optimal, as we demonstrate in the following lemma.

\begin{lemma}\label{lem:compl_optimality}
    Let $G \not \cong K_n, \overline{K_n}$ be a regular graph on $n$ vertices with adjacency eigenvalues $\lambda_1 \geq \cdots \geq \lambda_n$.
    Then, the optimal solution of SDP \eqref{eq:SDP_adjacency} has objective value 
    \begin{equation*}
     \frac{2m - \lambda_1(\lambda_1 - \lambda_2)}{\lambda_2 - \lambda_n}. 
    \end{equation*}
\end{lemma}
\begin{proof}
    Let $\mu_i'$ and $\mu_i$ be defined as in the proof of Theorem \ref{thm:energy_regular_scaled_complement}, and recall that the feasible solutions for SDP \eqref{eq:SDP_adjacency} are exactly the values $a$ and $b$ such that $\mu_1 - \mu_n \leq 1$, we write this restriction as the Linear Program (LP)
     \eqref{eq:LP_scaled_complement}, that has the same optimal solution as SDP \eqref{eq:SDP_adjacency}.

    Let $c_{ij}$ and $d_{ij}$ be such that $\mu_i' - \mu_j' = ac_{ij} + bd_{ij}$. That is, $c_{ij} = \lambda_i-\lambda_j$, $d_{ij} = \lambda_j - \lambda_i$ for $i \neq 1$ and $d_{1j} = n-\lambda_1+\lambda_j$. Note that $d_{ij} = -c_{ij}$ for $i \neq 1$ and that $c_{1j}+d_{1j} = n$.
   \begin{equation}\label{eq:LP_scaled_complement}
        \begin{aligned}
        \text{maximize}  \ \ &2ma \\
        \text{subject to}\ \
        &|ac_{12} + bd_{12}|\leq 1,\\
        &|ac_{1n} + bd_{1n}|\leq 1,\\
        &|ac_{2n} + bd_{2n}|\leq 1,\\
        & a, b \in \mathbb{R}.
        \end{aligned}
    \end{equation}
     
    Consider the dual of LP \eqref{eq:LP_scaled_complement}:
    \begin{equation}\label{eq:LP_dual_scaled_complement}
        \begin{aligned}
        \text{minimize}  \ \ & x_1 + x_2 + x_3 + x_4 + x_5 + x_6 \\
        \text{subject to}\ \
        &(x_1-x_2)c_{12} + (x_3-x_4)c_{1n} + (x_5-x_6)c_{2n} = 2m,\\
        &(x_1-x_2)d_{12} + (x_3-x_4)d_{1n} + (x_5-x_6)d_{2n} = 0,\\
        & x_i \geq 0.
        \end{aligned}
    \end{equation}
    LP \eqref{eq:LP_dual_scaled_complement} has feasible solution 
    \begin{equation*}
        x_i = \begin{cases}
            \frac{2m}{n} & \text{if } i = 3,\\
            \frac{2m}{n}\frac{d_{1n}}{c_{2n}} = \frac{-x_3d_{1n}}{d_{2n}} & \text{if } i = 5,\\
            0 & \text{else.}
        \end{cases}
    \end{equation*}

    With objective value 
    \begin{equation*}
        x_3 + x_5 = \frac{2m}{n} \cdot \frac{c_{2n} + d_{1n}}{c_{2n}} = \frac{2m - \lambda_1(\lambda_1 - \lambda_2)}{\lambda_2 - \lambda_n}.
    \end{equation*}

    Finally, the lemma holds by weak duality of linear programs.
\end{proof}

The power of Theorem \ref{thm:energy_regular_scaled_complement} comes from only needing limited knowledge of the spectrum of a graph ($\lambda_1,\lambda_2$ and $\lambda_n$), whereas using the energy explicitly requires much more information on the spectrum. To show how this can be useful for tackling Conjecture \ref{con:energy_lb}, consider the following result, where we show the conjecture holds for Johnson graphs.

\begin{corollary}\label{cor:Johnson}
    Conjecture \ref{con:energy_lb} holds for Johnson graphs $J(r, k)$. 
\end{corollary}
\begin{proof}
     The Johnson graph $J(r,k)$ has adjacency eigenvalues $(k-j)(r-k-j)-j$ for $j = 0,\hdots,\min\{k,r-k\}$. Since $J(r, k) \cong J(r, r-k)$, we assume $k \leq r-k$. So, $\lambda_1 = k(r-k),\lambda_2 = k(r-k)-r, \lambda_{n} = -k$ and $J(r,k)$ has $\binom{r}{k}$ vertices. 
     
     If $r \geq 5$ and $k \geq 2$, then
     \begin{equation*}
         \frac{2m- \lambda_1(\lambda_1 + \lambda_2)}{\lambda_2 - \lambda_n} = \frac{k(n-r)}{k-1} \geq n \iff \binom{r}{k} \geq rk.
     \end{equation*}
     Where the last inequality holds if $r \geq 5$ and $k \geq 2$. We can then use Theorem \ref{thm:energy_regular_scaled_complement} to conclude $\frac{1}{2}\mathcal{E}(J(r, k)) \geq n$.
     
     If $k = 1$, then $J(r, k)$ is a complete graph, which satisfies Conjecture \ref{con:energy_lb}. The remaining cases for $r < 5$ can be checked computationally.
\end{proof}

We note that Johnson graphs are not perfect in general, and the ratio bound is not always tight. Hence, neither Theorem \ref{thm:energy_n_minus_chi_f} nor Lemma \ref{lem:energy_bound_adjacency} can be used to prove Conjecture \ref{con:energy_lb} for Johnson graphs.

Moreover, Theorem \ref{thm:energy_regular_scaled_complement} can be seen as a lower bound for the graph energy that is interesting by itself, as it refines the following result due to Nikiforov \cite{nikiforov2007energy}. 

Let $M$ be a $m\times n$ matrix with complex entries, let $M^*$ be the Hermitian adjoint of $M$ and let $\|M\|_2^2 = \text{tr}(MM^*)$. The singular values $\sigma_1(M) \geq \dots \geq \sigma_m(M)$ of a matrix $M$ are the square roots of the eigenvalues of $MM^*$. The energy of $M$ is defined as $\mathcal{E}(M) = \sigma_1(M) + \dots + \sigma_m(M)$.
\begin{proposition}[\cite{nikiforov2007energy}, Eq.~(4)]\label{prop:nikiforov_lb}
    Let $M$ be a nonconstant matrix, then 
\begin{equation*}
    \mathcal{E}(M) \geq \sigma_1(M) + \frac{\|M\|_2^2 - \sigma_1^2(M)}{\sigma_2(M)}.
\end{equation*}
\end{proposition}

Nikiforov applies Proposition \ref{prop:nikiforov_lb} in the case $M$ is the adjacency matrix of a random graph to give a lower bound to the asymptotic behavior of the energy of almost all graphs. We observe that by restricting $M$ to be the adjacency matrix of a regular graph, Theorem \ref{thm:energy_regular_scaled_complement} is, in fact, a refinement of Proposition \ref{prop:nikiforov_lb}.

\begin{lemma} \label{lem:refinement_nikiforov_lb}
    Let $G \not \cong K_n, \overline{K_n}$ be a regular graph on $n$ vertices and $m$ edges with adjacency matrix $A$ and eigenvalues $\lambda_1 \geq \dots \geq \lambda_n$. Then
    \begin{equation*}
    \frac{2m - \lambda_1(\lambda_1 - \lambda_2)}{\lambda_2 - \lambda_n} \geq 
    \frac{\|A\|_2^2 - \sigma_1(A)(\sigma_1(A) - \sigma_2(A))}{2\sigma_2(A)}.
\end{equation*}
\end{lemma}
\begin{proof}
Note that under our assumption of $G$, $\|A\|_2^2 = 2m = \lambda_1 n$, $\sigma_1(A) = \lambda_1$ and $\sigma_2(A) = \max\{|\lambda_2|, |\lambda_n|\}$. If $\sigma_2(A) = |\lambda_2|$, then the inequality follows easily. Assume that $\sigma_2(A) = |\lambda_n| > |\lambda_2|$, we check that the difference between the two bounds is non-negative.
\begin{align*}
    \frac{2m - \lambda_1(\lambda_1 - \lambda_2)}{\lambda_2 + |\lambda_n|} -  \frac{2m - \lambda_1(\lambda_1 - |\lambda_n|)}{2|\lambda_n|}
    = \frac{\lambda_1 (|\lambda_n|-\lambda_2)(n-\lambda_1-|\lambda_n|)}{2|\lambda_n|(\lambda_2+|\lambda_n|)} \geq 0.
\end{align*}
Where the inequality follows since $\lambda_1 \geq 0$, we are assuming $|\lambda_n| > |\lambda_2|$ and for regular graphs $n \geq \lambda_1 + |\lambda_n|$ (see e.g.~\cite{gregory2001spread}).
\end{proof}

We finally note that while the main Theorems proven in this section (Theorems \ref{thm:energy_n_minus_ratio} and \ref{thm:energy_regular_scaled_complement}) only apply to regular graphs, SDP \eqref{eq:SDP_adjacency} can be optimized computationally for non-regular graphs. We used cvxpy in SageMath to computationally solve the SDPs and compare the performance of our various approaches for several SageMath named graphs, and all small graphs on $n$ vertices. The results of these experiments can be found in the Appendix in Tables \ref{tab:named_comp} and \ref{tab:small_comp}.

%%%%%%%%%%%%%%%%%%%%%%%%%%%%%%%%%%%%%%%%%%%%
\section{Highly regular graphs and Schrijver's theta number} \label{sec:high_reg}
%%%%%%%%%%%%%%%%%%%%%%%%%%%%%%%%%%%%%%%%%%%%

In this section, we provide a proof of Theorem \ref{thm:highly_regular}, which gives a stronger bound in support of Conjecture \ref{con:energy_lb} for the so-called highly regular graphs.

\begin{definition}
    Let $G$ be a graph with adjacency matrix $A$. We say that $G$ is \emph{highly regular} if there exists a matrix algebra $\mathcal{A}$ satisfying the properties:
    \begin{enumerate} \label{prop:algebra}
        \item $A, I, J \in \mathcal{A}$;
        \item The elements of $\mathcal{A}$ have constant diagonal;
        \item There exists a transpose-closed 01-basis $\{B_1, \dots, B_k\}$ for $\mathcal{A}$ that splits $A$, that is, $\exists I :A = \sum_{i \in I}B_i$ and $A\circ B_j = 0, \,  \forall j \not \in I$.
    \end{enumerate}
\end{definition}
The class of highly regular graphs encompasses several well-known graph classes, such as distance regular graphs, vertex transitive graphs and 1-walk regular graphs, see \cite{de2019algebras}.

In order to prove Theorem \ref{thm:highly_regular} we first need some preliminary results on highly regular graphs.

\begin{lemma}[\cite{de2019algebras}, Corollary 2.2] \label{lem:algebra_psd}
   Let $\mathcal{A}$ be an algebra satisfying the conditions in Definition \ref{prop:algebra}. Let $M$ be a positive semidefinite matrix, then the projection of $M$ in $\mathcal{A}$ is positive semidefinite.
\end{lemma}

Although the main result of this section is a bound related to Schrijver's theta, we first focus our attention on \emph{Szegedy's theta}, denoted $\vartheta^+(G)$, which is defined as the solution of SDP \eqref{eq:SDP_szegedy}.
\begin{equation}\label{eq:SDP_szegedy}
    \begin{aligned}
        \text{maximize}  \ \ &\langle J,Y \rangle \\
        \text{subject to}\ \ &Y  \succcurlyeq 0, \\
        &\langle I, Y \rangle = 1, \\
        &Y \circ A \leq 0.
    \end{aligned}
\end{equation}

\begin{lemma}[\cite{de2019algebras}, Corollary 10] \label{lem:schrijver_szegedy_n}
   Let $G$ be a $n$-vertex highly regular graph. Then, $$\vartheta^-(G)\vartheta^+(G) = n.$$
\end{lemma}

\begin{lemma}\label{lem:sdp_opt_in_algebra}
    Let $G$ be a highly regular graph with matrix algebra $\mathcal{A}$. There exists a matrix $Y' \in \mathcal{A}$ that is an optimal solution for SDP \eqref{eq:SDP_szegedy}.
\end{lemma}
\begin{proof}
 Let $\mathcal{A}$ be the algebra satisfying the properties from Definition \ref{prop:algebra}. Let $Y$ be an optimal solution for SDP \eqref{eq:SDP_szegedy}; we claim that the projection $Y'$ onto $\mathcal{A}$ is also an optimal solution.

It follows from Lemma \ref{lem:algebra_psd} that $Y' \succcurlyeq 0$. Since $I$ belongs to $\mathcal{A}$ and all matrices of the algebra have constant diagonal, we may assume that $I$ belongs to the $0/1$-basis and thus $\langle Y', I\rangle = \langle Y, I\rangle = 1$. Let $B_1, \dots, B_t$ be $0/1$-matrices on a basis such that $\sum_i B_i = A$, as $Y \circ A \leq 0$, we have that $\langle Y, B_i \rangle \leq 0$ and therefore, $Y' \circ A  = \sum_i \langle Y, B_i \rangle B_i\leq 0$. So $Y'$ is feasible.

Since $J$ belongs to $\mathcal{A}$, we have that $\langle Y', A\rangle = \langle Y, A\rangle$, and we conclude that $Y'$ is an optimal solution for SDP \eqref{eq:SDP_szegedy}.
\end{proof}

Our strategy for proving Theorem \ref{thm:highly_regular} is to construct a feasible solution $X$ to SDP \eqref{eq:SDP_max}. The following key lemma will be used for this purpose.

\begin{lemma}\label{lem:key_lemma_highly_regular}
    Let $G$ be a highly regular graph with matrix algebra $\mathcal{A}$. Let $Y \in \mathcal{A}$ be an optimal solution for SDP \eqref{eq:SDP_szegedy} with largest eigenvalue $\lambda_1(Y)$. Then, $\lambda_1(Y) \leq 1/\vartheta^-(\overline{G})$.
\end{lemma}
\begin{proof}
    The proof follows by contradiction. Assume that $\lambda_1(Y) = \mu > 1/\vartheta^{-}(\overline{G})$. We show that there exists a matrix $Y'$, feasible for SDP \eqref{eq:SDP_szegedy} with objective value greater than $\vartheta^{+}(G)$.

    Let $E$ be the projector of $Y$ on the $\mu$ eigenspace. Note that since $Y \in \mathcal{A}$, $E$ also belongs to $\mathcal{A}$, as it is a polynomial in $Y$. Therefore, $E$ has constant diagonal. Let $n = |V(G)|$, define $\overline{E} = nE/\text{tr}E$ and let $c \geq 0$ be sufficiently large so that $\overline{E} + cJ \geq 0$.

    Define the matrix $Y' := Y \circ \left(\overline{E} + cJ\right) \frac{1}{c+1}$, we show that $Y'$ is feasible for SDP \eqref{eq:SDP_szegedy}:

    \begin{enumerate}
        \item $Y' \succcurlyeq 0$, because $Y, E, J \succcurlyeq0$ and the cone of PSD matrices is closed under Schur product and addition.
        \item $\langle I, Y'\rangle = \sum_{i} Y_{ii} \left(\overline{E}_{ii} + c\right)\frac{1}{c+1} = \sum_{i} Y_{ii} = 1$, because $\overline{E}$ has constant diagonal equal to $\mathbf{1}$.
        \item $(Y' \circ A)_{ij} = Y_{ij}A_{ij} (\overline{E}_{ij} + c) \frac{1}{c+1} \leq 0$, because $Y_{ij}A_{ij} \leq 0$ and $\overline{E}_{ij} + c \geq 0$, by choice of $c$.
    \end{enumerate}

    And we evaluate $\langle J, Y' \rangle$:
    \begin{align*}
    \langle J, Y' \rangle &= \frac{1}{c+1} \left\langle Y, E \frac{n}{\mathrm{Tr}E} + cJ \right\rangle \\
    &= \frac{n}{c+1} \left(\mu + \frac{c}{\vartheta^{-}(\overline{G})}\right) \\
    &> \frac{n}{\vartheta^{-}(\overline{G})} \\
    &= \vartheta^{+}(G).
    \qedhere \end{align*}
\end{proof}

We now have all the ingredients to prove Theorem \ref{thm:highly_regular}.

\begin{proof}[Proof of Theorem \ref{thm:highly_regular}]
    Let $G$ be a highly regular graph with matrix algebra $\mathcal{A}$. Let $Y \in \mathcal{A}$ be an optimal solution for SDP \eqref{eq:SDP_szegedy} for $\overline{G}$, which exists due to Lemma \ref{lem:sdp_opt_in_algebra}.

    Let $X = \vartheta^-(G)\cdot Y$. Since $Y$ is positive semidefinite, $X$ is also positive semidefinite. Moreover, due to Lemma \ref{lem:key_lemma_highly_regular}, $\lambda_1(X) = \vartheta^-(G) \lambda_1(Y) \leq 1$, so $I - X$ is also positive semidefinite. Therefore, $X$ is a feasible solution for SDP \eqref{eq:SDP_max} with objective value:
    \begin{align*}
        \langle A, X \rangle &= \vartheta^-(G)  (\langle J,Y\rangle - \langle \overline{A},Y\rangle-1) \\
        &\geq \vartheta^-(G)\vartheta^+(\overline{G}) - \vartheta^-(G) \\
        &= n-\vartheta^-(G).
\qedhere    \end{align*}
\end{proof}

%%%%%%%%%%%%%%%%%%%%%%%%%%%%%%%%%%%%%%%%%%%%%%%%%%%%%%%%%%%%%%%%%%%%%%%%%%%%
\section{Concluding remarks}\label{sec:concludingremarks}
%%%%%%%%%%%%%%%%%%%%%%%%%%%%%%%%%%%%%%%%%%%%%%%%%%%%%%%%%%%%%%%%%%%%%%%%%%%%%
% ===========================================

In this work, we devised an SDP approach for tackling Conjecture~\ref{con:energy_lb}. Using the formulation presented in Theorem~\ref{thm:SDP_max}, we obtained several new bounds on the energy (Theorems \ref{thm:energy_n_minus_chi_f}, \ref{thm:energy_n_minus_ratio} and \ref{thm:energy_regular_scaled_complement}) which improve on prior results by Hoffman \cite{hoffman2003eigenvalues} and Nikiforov \cite{nikiforov2007energy}. Additionally, we strengthened both bounds for the class of highly regular graphs (Theorem \ref{thm:highly_regular}). Finally, we note that our method for proving Theorem \ref{thm:energy_n_minus_chi_f} makes use of the new Lemma \ref{lem:convex_comb_orth_matrices}, which is interesting in its own right and improves results by Andrade at. al. \cite[Theorem 1.1]{andrade2017lower} and by Day and So \cite[Theorem 3.4]{day2008graph}.

\medskip

We end with some future research directions.
\begin{itemize}
    \item Denote by $\xi_f(G)$ the \emph{projective rank} of $G$. Based on the results from Elphick et al.~\cite{elphick2025spectrallowerboundchromatic}, we expect that Lemma \ref{lem:convex_comb_orth_matrices}, used to derive the lower bound $n-\chi_f(\overline{G})$ on the graph energy, may also prove useful in proving the following weakening of Conjecture~\ref{con:energy_lb}:
\begin{conjecture}
    \label{que:xi_f}
    Let $G$ be a $n$-vertex graph, then
    \begin{equation*}
        \frac{1}{2} \mathcal{E}(G) \geq n-\xi_f(\overline{G}).
    \end{equation*}
\end{conjecture}
    \item Another approach to tackle Conjecture \ref{con:energy_lb} in general could go through induction. If we partition a graph $G$ into $G_1,G_2$, then it holds that $\mathcal{E}(G) \geq \mathcal{E}(G_1)+\mathcal{E}(G_2)$. However, it is quite hard to control $\alpha(G_1)$ and $\alpha(G_2)$, furthermore, there exist graphs such as complete graphs, odd cycles, and the Moser Spindle graph satisfying $\mathcal{E}(G_1)+\mathcal{E}(G_2) < n-\alpha(G)$, for all non-trivial partitions. 
\end{itemize}

%%%%%%%%%%%%%%%%%%%%%%%%%%%%%%%%%%%%%%%%%%%%%%%%%%%%%%%%%%%%%%%%%%
\subsection*{Acknowledgments}
%%%%%%%%%%%%%%%%%%%%%%%%%%%%%%%%%%%%%%%%%%%%%%%%%%%%%%%%%%%%%%%%%%

Aida Abiad is supported by NWO (Dutch Research Council) through the grants VI.Vidi.213.085 and OCENW.KLEIN.475. Luuk Reijnders is supported through the grant VI.Vidi.213.085. Gabriel Coutinho and Emanuel Juliano are supported by FAPEMIG and CNPq grants.

%%%%%%%%%%%%%%%%%%%%%%%%%%%%%%%%%%%%%%%%%%%%%%%%%%%%%%%%%%%%%%%%%%
\bibliographystyle{abbrv}
\bibliography{references.bib}
%%%%%%%%%%%%%%%%%%%%%%%%%%%%%%%%%%%%%%%%%%%%%%%%%%%%%%%%%%%%%%%%%%

 \newpage
%%%%%%%%%%%%%%%%%%%%%%%%%%%%%%%%%%%%%%%%%%%%
    \section*{Appendix: computational experiments}\label{sec:computational}
%%%%%%%%%%%%%%%%%%%%%%%%%%%%%%%%%%%%%%%%%%%%%

\begin{table}[H]
    \centering
    \tiny
        \begin{tabular}{l|ccc|c}
        Graph & SDP \eqref{eq:SDP_adjacency} ($b=0$) & Theorem \ref{thm:energy_n_minus_chi_f} & SDP \eqref{eq:SDP_adjacency} & SDP \eqref{eq:SDP_max} \\
        \hline
        Balaban 10-cage & $1.0$ & $1.0$ & $0.91544$ & $0.64477$ \\
        Balaban 11-cage & $1.04568$ & $1.07143$ & $0.96127$ & $0.71395$ \\
        Bidiakis cube & $1.08142$ & $1.16667$ & $0.9676$ & $0.86174$ \\
        Biggs-Smith graph & $1.06664$ & $1.15686$ & $0.98635$ & $0.77505$ \\
        Blanusa First Snark Graph & $1.03898$ & $1.11111$ & $0.95707$ & $0.70215$ \\
        Blanusa Second Snark Graph & $1.09535$ & $1.22222$ & $1.00994$ & $0.7796$ \\
        Brinkmann graph & $1.14076$ & $1.33333$ & $0.89242$ & $0.72865$ \\
        Brouwer-Haemers & $1.1$ & $1.22222$ & $0.47143$ & $0.47143$ \\
        Bucky Ball & $1.12361$ & $1.2$ & $1.07931$ & $0.77285$ \\
        Cell 600 & $1.04721$ & $1.06667$ & $0.91184$ & $0.61114$ \\
        Chvatal graph & $1.16666$ & $1.33333$ & $0.95414$ & $0.83668$ \\
        Clebsch graph & $1.1$ & $1.375$ & $0.73333$ & $0.73333$ \\
        Coclique graph of Hoffmann-Singleton graph & $1.0$ & $1.0$ & $0.74075$ & $0.41666$ \\
        Conway-Smith graph for 3S7 & $1.0$ & $1.07143$ & $0.69827$ & $0.53572$ \\
        Coxeter Graph & $1.03128$ & $1.14286$ & $0.87194$ & $0.74469$ \\
        Desargues Graph & $1.0$ & $1.0$ & $0.8772$ & $0.625$ \\
        Dejter Graph & $1.0$ & $1.0$ & $0.8771$ & $0.48607$ \\
        Dodecahedron & $1.04721$ & $1.2$ & $0.92994$ & $0.81587$ \\
        Double star snark & $1.05458$ & $1.13333$ & $0.98618$ & $0.72948$ \\
        Durer graph & $1.16357$ & $1.14286$ & $1.06138$ & $0.88259$ \\
        Dyck graph & $1.0$ & $1.0$ & $0.89402$ & $0.62951$ \\
        Ellingham-Horton 54-graph & $1.0$ & $1.0$ & $0.98243$ & $0.65042$ \\
        Ellingham-Horton 78-graph & $1.0$ & $1.0$ & $0.99441$ & $0.64273$ \\
        Errera graph & $1.05173$ & $1.05882$ & $0.89508$ & $0.719$ \\
        F26A Graph & $1.0$ & $1.0$ & $0.87694$ & $0.63552$ \\
        Flower Snark & $1.04879$ & $1.1$ & $0.9209$ & $0.69558$ \\
        Folkman Graph & $1.0$ & $1.0$ & $0.87394$ & $0.72475$ \\
        Foster Graph & $1.0$ & $1.0$ & $0.91384$ & $0.65795$ \\
        Foster graph for 3.Sym(6) graph & $1.0$ & $1.0$ & $0.71428$ & $0.58824$ \\
        Franklin graph & $1.0$ & $1.0$ & $0.88185$ & $0.63398$ \\
        Frucht graph & $1.03808$ & $1.0$ & $0.95425$ & $0.77462$ \\
        Goldner-Harary graph & $0.97521$ & $1.0$ & $0.71042$ & $0.53546$ \\
        Golomb graph & $1.16577$ & $1.0$ & $0.94906$ & $0.72263$ \\
        Gosset Graph & $1.03174$ & $1.08333$ & $0.60819$ & $0.57778$ \\
        Gray graph & $1.0$ & $1.0$ & $0.9176$ & $0.70164$ \\
        Gritsenko strongly regular graph & $1.03621$ & $1.08923$ & $0.40691$ & $0.40691$ \\
        Grotzsch graph & $0.98724$ & $1.09091$ & $0.75468$ & $0.68956$ \\
        Hall-Janko graph & $1.0$ & $1.2$ & $0.35715$ & $0.35715$ \\
        Harborth Graph & $1.08702$ & $1.0$ & $1.0645$ & $0.75935$ \\
        Harries Graph & $1.0$ & $1.0$ & $0.91544$ & $0.64182$ \\
        Harries-Wong graph & $1.0$ & $1.0$ & $0.91544$ & $0.64182$ \\
        Heawood graph & $1.0$ & $1.0$ & $0.82968$ & $0.60946$ \\
        Herschel graph & $0.93501$ & $1.0$ & $0.82489$ & $0.63474$ \\
        Hexahedron & $1.0$ & $1.0$ & $0.88889$ & $0.66668$ \\
        Higman-Sims graph & $1.06364$ & $1.56$ & $0.44319$ & $0.44319$ \\
        Hoffman Graph & $1.0$ & $1.0$ & $0.85714$ & $0.66666$ \\
        Hoffman-Singleton graph & $1.0$ & $1.4$ & $0.55555$ & $0.55556$ \\
        Holt graph & $1.03915$ & $1.25926$ & $0.81822$ & $0.72004$ \\
        Horton Graph & $1.0$ & $1.0$ & $0.99092$ & $0.64086$ \\
        Icosahedron & $1.08541$ & $1.125$ & $0.87156$ & $0.76869$ \\
        Kittell Graph & $1.08107$ & $1.04348$ & $0.91998$ & $0.7189$ \\
        Klein 3-regular Graph & $1.09245$ & $1.17857$ & $0.98771$ & $0.7908$ \\
        Klein 7-regular Graph & $1.03347$ & $1.125$ & $0.6926$ & $0.63906$ \\
        Krackhardt Kite Graph & $1.47121$ & $1.0$ & $1.08809$ & $0.75307$ \\
        Ljubljana graph & $1.0$ & $1.0$ & $0.93057$ & $0.66179$ \\
        M22 Graph & $1.0$ & $1.45455$ & $0.44444$ & $0.44446$ \\
        Markstroem Graph & $1.08306$ & $1.04651$ & $1.0271$ & $0.80824$ \\
        McGee graph & $1.08142$ & $1.16667$ & $0.92554$ & $0.78581$ \\
        Meredith Graph & $1.01474$ & $1.02857$ & $0.9915$ & $0.74437$ \\
        Moebius-Kantor Graph & $1.0$ & $1.0$ & $0.85655$ & $0.6188$ \\
        Moser spindle & $1.17762$ & $1.11111$ & $1.03853$ & $0.8948$ \\
        Murty Graph & $1.08786$ & $1.07143$ & $1.01408$ & $0.78328$ \\
        Nauru Graph & $1.0$ & $1.0$ & $0.86957$ & $0.66667$ \\
        Octahedron & $1.0$ & $1.0$ & $1.0$ & $1.0$ \\
        Pappus Graph & $1.0$ & $1.0$ & $0.84844$ & $0.67204$ \\
        Perkel Graph & $1.0$ & $1.33333$ & $0.6636$ & $0.63333$ \\
        Petersen graph & $1.0$ & $1.2$ & $0.75$ & $0.75$ \\
        Poussin Graph & $1.04016$ & $1.0$ & $0.85787$ & $0.70888$ \\
        Robertson Graph & $1.07231$ & $1.26316$ & $0.81874$ & $0.68174$ \\
        Schläfli graph & $1.0$ & $1.06667$ & $0.6$ & $0.60001$ \\
        Shrikhande graph & $1.0$ & $1.125$ & $0.66666$ & $0.66667$ \\
        Sims-Gewirtz Graph & $1.0$ & $1.42857$ & $0.5$ & $0.5$ \\
        Sousselier Graph & $1.13844$ & $1.25$ & $0.92397$ & $0.77303$ \\
        Sylvester Graph & $1.06666$ & $1.33333$ & $0.72727$ & $0.64865$ \\
        Szekeres Snark Graph & $1.08457$ & $1.16$ & $1.03988$ & $0.74482$ \\
        Tetrahedron & $1.0$ & $1.0$ & N/A & $1.0$ \\
        Thomsen graph & $1.0$ & $1.0$ & $1.0$ & $1.0$ \\
        Tietze Graph & $1.0311$ & $1.07692$ & $0.91271$ & $0.72875$ \\
        Tricorn Graph & $1.0$ & $1.0$ & $0.87379$ & $0.77332$ \\
        Truncated Tetrahedron & $1.11111$ & $1.0$ & $0.9697$ & $0.88889$ \\
        Tutte 12-Cage & $1.0$ & $1.0$ & $0.91223$ & $0.68018$ \\
        Tutte-Coxeter graph & $1.0$ & $1.0$ & $0.86207$ & $0.71428$ \\
        Twinplex Graph & $1.08142$ & $1.16667$ & $0.91091$ & $0.74549$ \\
        Wagner Graph & $1.12796$ & $1.25$ & $0.94839$ & $0.85786$ \\
        Wells graph & $1.1$ & $1.375$ & $0.78803$ & $0.66892$ \\
        Wiener-Araya Graph & $1.11976$ & $1.19048$ & $1.07161$ & $0.77172$ \\
        \hline
    \end{tabular}  
    \caption{$n-\alpha(G)$ divided by the bound obtained from the various approaches. If at most $1$, the approach verifies Conjecture \ref{con:energy_lb}.}
    \label{tab:named_comp}
\end{table}

    \begin{table}[H]
        \centering
         \begin{tabular}{c|ccc|c}
         \hline
        $n$ & SDP \eqref{eq:SDP_adjacency} ($b=0$) & Theorem \ref{thm:energy_n_minus_chi_f} & SDP \eqref{eq:SDP_adjacency} &  Total\\
        \hline
        $1$ & $0.0\%$ & $0.0\%$ & $0.0\%$ & $0.0\%$ \\
        $2$ & $50.0\%$ & $50.0\%$ & $0.0\%$ & $50.0\%$\\
        $3$ & $50.0\%$ & $75.0\%$ & $25.0\%$ & $75.0\%$\\
        $4$ & $63.6\%$ & $90.9\%$ & $45.5\%$ & $90.9\%$\\
        $5$ & $44.1\%$ & $94.1\%$ & $44.1\%$ & $97.1\%$\\
        $6$ & $30.7\%$ & $96.8\%$ & $40.3\%$ & $97.4\%$\\
        $7$ & $16.2\%$ & $96.3\%$ & $40.5\%$ & $96.6\%$\\
        $8$ & $7.5\%$ & $95.6\%$ & $39.8\%$ & $95.9\%$\\
        $9$ & $3.0\%$ & $93.8\%$ & $39.9\%$ & $94.1\%$\\
        \hline
        \end{tabular}
        \caption{Success rates of the three approaches over all graphs of given order $n$.}
        \label{tab:small_comp}
    \end{table}
     \newpage
\end{document}